\documentclass{article}
\usepackage[utf8]{inputenc}

\usepackage[margin=2cm,a4paper]{geometry}
\usepackage{multirow,listings,setspace,gnuplottex,latexsym,keyval,ifthen,moreverb,lscape,forest}
\usepackage{pgf,tikz}
\usetikzlibrary{arrows}

\usepackage{amsmath, amssymb, amsthm, bm}
\usepackage{color}
\usepackage{url}
\usepackage{enumerate}
\usepackage{graphicx}
\usepackage{appendix}
\usepackage{cite}

\newtheorem{thm}{Theorem}
\newtheorem{prop}[thm]{Proposition}
\newtheorem{lem}[thm]{Lemma}

\newtheorem{conj}[thm]{Conjecture}
\numberwithin{thm}{section}
\theoremstyle{definition}

\theoremstyle{remark}

\title{Non-intersecting Ryser hypergraphs}
\author{Anurag Bishnoi
        \thanks{E-mail: {\tt anurag.2357@gmail.com}. Research supported by a Humboldt Research Fellowship for Postdoctoral Researchers}
        \and
        Valentina Pepe\thanks{E-mail: {\tt valentina.pepe@sbai.uniroma1.it}. The author acknowledges the support of Anvur-FFABR funding 2017.}}
\begin{document}
\maketitle
\begin{abstract}
    A famous conjecture of Ryser states that every $r$-partite hypergraph  has vertex cover number  at most $r - 1$ times the matching number.
    In recent years, hypergraphs meeting this conjectured bound, known as $r$-Ryser hypergraphs, have been studied extensively.
    It was proved by Haxell, Narins and Szab\'{o} that all $3$-Ryser hypergraphs with matching number $\nu > 1$ are essentially obtained by taking $\nu$ disjoint copies of intersecting $3$-Ryser hypergraphs.
    Abu-Khazneh showed that such a characterisation is false for $r = 4$ by giving a computer generated example of a $4$-Ryser hypergraph with $\nu = 2$ whose vertex set cannot be partitioned into two sets such that we have an intersecting $4$-Ryser hypergraph on each of these parts.
    Here we construct first infinite families of $r$-Ryser hypergraphs, for any fixed matching number $\nu > 1$, that are truly non-intersecting in the sense that they do not contain two vertex disjoint intersecting $r$-Ryser subhypergraphs.
\end{abstract}

\section{Introduction}
A hypergraph $\mathcal{H}$ is called $r$-partite if its vertex set $V(\mathcal{H})$ can be partitioned into $r$ pairwise disjoint sets $V_1, \dots, V_r$, called \textit{sides}, such that every edge $e \in E(\mathcal{H})$  satisfies $|e \cap V_i| = 1$ for all $1 \leq i \leq r$.
Clearly, an $r$-partite hypergraph is also $r$-uniform, that is, every edge of the hypergraph contains $r$ vertices.
Note that $2$-partite hypergraphs are equivalent to bipartite graphs.
A subset $B$ of $V(\mathcal{H})$ is known as a \textit{vertex cover}, or a \textit{blocking set}, if for all edges $e \in E(\mathcal{H})$ we have $e \cap B \neq \emptyset$.
The \textit{vertex cover number} $\tau(\mathcal{H})$ is the smallest cardinality of a vertex cover of $\mathcal{H}$.
A subset $M$ of $E(\mathcal{H})$ is called a \textit{matching} if $e \cap e' = \emptyset$ for all distinct $e, e'$ in $M$.
The \textit{matching number} $\nu(\mathcal{H})$ is the largest cardinality of a matching in $\mathcal{H}$.
Hypergraphs that have matching number equal to $1$ are known as \textit{intersecting hypergraphs}.
For any $r$-uniform hypergraph we have $\nu(\mathcal{H}) \leq \tau(\mathcal{H}) \leq r \nu(\mathcal{H})$,
where the lower bound follows from the fact that any vertex cover must contain at least one vertex from each edge in a maximum matching, and the upper bound follows from taking the union of all vertices in a maximum matching.
It can be easily shown that both of these bounds are tight for arbitrary hypergraphs.
However, for $r$-partite hypergraphs, Ryser conjectured in 1960's that the upper bound can be strengthened.
\begin{conj}
For any $r$-partite hypergraph $\mathcal{H}$,
\[\tau(\mathcal{H}) \leq (r - 1)\nu(\mathcal{H}).\]
\end{conj}
This conjecture first appeared in the Ph.D. thesis of Ryser's student, Henderson \cite{Henderson71} (see \cite{BW18} for more on the history of this conjecture).
For $r = 2$, the conjecture is true as it is equivalent to the well known K\H{o}nig's theorem on bipartite graphs \cite{Konig}, and for $r = 3$ it was proved by Aharoni \cite{Aharoni} using topological methods.
Except for these two cases, the conjecture is wide open.
If we assume that the hypergraph is also intersecting, then it is known to be true for $r \leq 5$, as proved by Tuza \cite{Tuza}, and if we further assume that the hyeprgraph is linear, that is, every two edges intersect in a unique vertex, then Franceti\'{c}, Herke, McKay, and Wanless \cite{FHMW17} have proved it for $r \leq 9$.

A \textit{finite projective plane} of order $n$ is a linear intersecting hypergraph with the following properties: (1) every edge contains $n + 1$ vertices and every vertex is contained in $n + 1$ edges;
    (2) any two distinct vertices are contained in a unique edge;
     (3) there exists a set of four vertices such that no three lie on a common edge.
The vertices and edges of a projective plane are typically referred to as points and lines, respectively.
It easily follows from the definition that a projective plane has $n^2+n+1$ points and $n^2+n+1$ lines.
If we remove a point from a projective plane of order $n$, and all $n + 1$ lines through that point, then we get an $(n + 1)$-partite intersecting hypergraph, with sides corresponding to the lines that were removed, and the vertex cover number of this hypergraph is $n$. In fact, we have the following folklore result, from which it follows that a cover of size $n$ must be a side.

\begin{lem}
\label{lem:blocking_basis}
Every blocking set in a projective plane of order $n$ has size at least $n+ 1$, with equality if and only if the blocking set is a line.
\end{lem}

The hypergraph that we get by removing a point from a  projective plane of order $n$ is known as the \textit{truncated projective plane}, and denoted by $\mathcal{T}_{n + 1}$.
Therefore, if Ryser's conjecture is true then it will be tight for all values of $r$ for which there exists a projective plane of order $r - 1$.
By taking disjoint copies of truncated projective planes one can obtain tight examples for any matching number.
Interestingly, projective planes of order $n$ are only known to exist when $n$ is a prime power (and it is a major open problem to determine whether the order of any finite projective plane must be a prime power).
The classical example of these objects, known as \textit{Desarguesian} projective planes, can be obtained by taking the $1$-dimensional subspaces of the vector space $\mathbb{F}_q^3$ as points and the $2$-dimensional subspaces as lines, where $\mathbb{F}_q$ is the finite field of order $q$. These projective planes are denoted by $\mathrm{PG}(2,q)$.
Note that there are several families of projective planes that are non isomorphic to the Desarguesian planes.

The $r$-partite hypergraphs $\mathcal{H}$ that satisfy $\tau(\mathcal{H}) \geq (r - 1) \nu(\mathcal{H})$ are known as \textit{$r$-Ryser hypergraphs}. 
Until recently, truncated projective planes were the only known examples of $r$-Ryser hypergraphs.
A new infinite family of $r$-Ryser hypergraphs was constructed by Abu-Khazneh, Bar\'{a}t, Pokrovskiy and Szab\'{o} \cite{ABPS19}, for $r - 2$ equal to an arbitrary prime power.
Besides these values of $r$, the only other case for which we know the existence of $r$-Ryser hypergraphs is when both $(r - 1)/2$ and $(r + 1)/2$ are prime powers, which is due to Haxell and Scott \cite[Sec.~5]{Haxell-Scott17} (see also \cite{ABW16} and \cite{FHMW17} for some small examples).
Finite projective and affine planes play a crucial role in both of these constructions.
As mentioned ealier, the existence of these intersecting $r$-Ryser hypergraphs also implies the existence of $r$-Ryser hypergraphs with arbitrary matching number, by simply taking disjoint copies.
Moreover, it was proved by Haxell, Narins and Szab\'o \cite{HNS18} that for $r = 3$ every $r$-Ryser hypergraph $\mathcal{H}$ contains $\nu(\mathcal{H})$ many disjoint copies of intersecting $r$-Ryser hypergraphs.
Abu-Khazneh obtained computer generated examples which showed that such a result cannot hold true for $r = 4$ \cite[Ch.5]{A16}, and in \cite{ABPS19} it was asked how the characterisation of the $3$-Ryser hypergraphs can be generalised to higher values of $r$.
Here, we construct non-intersecting $r$-Ryser hypergraphs which prove that this characterisation fails for infinitely many values of $r$.

\begin{thm}
For every $\nu > 1$ and $r > 3$, with $r - 1$ a prime power, there exist an $r$-partite hypergraph $\mathcal{H}$ with $\nu(\mathcal{H}) = \nu$ and $\tau(\mathcal{H}) = (r - 1)\nu$ such that $\mathcal{H}$ does not contain two vertex disjoint subhypergraphs that are intersecting $r$-Ryser hypergraphs.
\end{thm}

Our first construction works whenever $r - 1$ is an odd prime, while our second construction works whenever $r - 1$ is any prime power greater than or equal to $4$.
The first family contains one of the examples obtained by Abu-Khazneh, which will be discussed in the appendix.

\section{First construction}
We will first construct hypergraphs with matching number $2$, prove our main result on these hypergraphs, and then give the general construction in the end (whose proof is in fact similar and left to the reader).

We will use an intersecting Ryser hypergraph that is a subgraph of the truncated projective plane, obtained from lines of the plane intersecting a fixed \textit{conic}.
A conic $\mathcal{C}$ of PG$(2,q)$ is a set of points satisfying a non-degenerate quadratic equation.
An \textit{arc} of PG$(2,q)$ is a set $S$ of points such that three of them are never collinear (contained in a common line).
Since there are $q + 1$ lines through each point, it is easy to see that $|S| \leq q+2$.
Moreover, we can only have $|S|=q+2$ for even $q$.
From the definition of a conic it follows that every conic is an arc of size $q + 1$.
In fact, if $q$ is odd then the converse also holds true by a famous theorem of Segre \cite{Segre}.
The lines intersecting a conic $\mathcal{C}$ in exactly one point are called \textit{tangent} to the conic and through every point of $\mathcal{C}$ there is exactly one tangent.
The lines meeting the conic in two points are called \textit{secant} and the lines disjoint from the conic are called \textit{external} lines.
For even $q$, there exists a point $N$, called the nucleus of $\mathcal{C}$, such that every line through $N$ is tangent to $\mathcal{C}$, and hence through every point $P \neq N$ not in $\mathcal{C}$ there is only one tangent line: $PN$. For odd $q$, if $P \notin \mathcal{C}$, then through $P$ there are two or zero tangents to $\mathcal{C}$. For an overview of this topic, see \cite[Ch.~7, 8]{Hbook}.

In the truncated projective plane $\mathcal{T}_{q+1}$, corresponding to PG$(2,q)$, let $\mathcal{C}$ be a conic through the point of the plane we have removed, say $Q$.
Let $\mathcal{C}'=\mathcal{C}\setminus \{Q\}$ and let $\mathcal{TC}_{q+1}$ be the subhypergraph formed by the lines intersecting $\mathcal{C}'$ in at least a point. This hypergraph is clearly $(q+1)$-partite and intersecting. It is well known that the cover number is also $q$ (see for example \cite{ABW16}), but we shall include a short proof for sake of completeness.
Let $B$ a blocking set of the lines meeting $\mathcal{C}$.
If $\mathcal{C}\subseteq B$, then $|B|\geq q+1$.
Suppose than $\mathcal{C} \not\subseteq B$, then there exists a point $P\in \mathcal{C}, P \notin B $. The $q+1$ lines through $P$ have to be blocked by $B$ and as $P \notin B$, we need at least $q+1$ points to block them, hence $|B|\geq q+1$. If $B'$ is a blocking set for the lines meeting $\mathcal{C}'$, then $B=B' \cup \{Q\} $ is blocking set for the lines meeting $\mathcal{C}$ and hence $|B'|\geq q$.
We will need the following characterization of minimum blocking sets with respect to lines meeting a conic.

\begin{prop}[Boros-F\"{u}redi-Kahn{\cite[Sec.~2]{BFK89}}]
\label{blockconic}
Let $B$ be a blocking set of size $q + 1$ for secant and tangent lines to a conic $\mathcal{C}$ in PG$(2,q)$. Then one of the following is true:
\begin{enumerate}[$(1)$]
    \item $B$ is a line;
    \item $B = \mathcal{C}$;
    \item $B=(\mathcal{C}\setminus S_1) \cup S_2$, where $S_1$ is a subset of the points of $\mathcal{C}$ not in a line $\ell$, $S_2$ is a subset of $\ell \setminus \mathcal{C}$, $|S_1|=|S_2|$ and it must be the size of a subgroup of $G$, where $G$ is a finite group of order $n=|\mathcal{C}\setminus \ell |$.
    Moreover, $G$ is cyclic if $n\in \{q-1,q+1\}$, and elementary abelian if $n=q$.
\end{enumerate}
\end{prop}

We can now describe our first hypergraph $\mathcal{H}_1$.
Let $q$ be a prime power and let $\pi_1$ and $\pi_2$ be two projective planes isomorphic to $\mathrm{PG}(2,q)$ with point sets $\mathcal{P}_1$ and $\mathcal{P}_2$, respectively, such that $\mathcal{P}_1 \cap \mathcal{P}_2 = \{P\}$.
Let $Q_1 \in \mathcal{P}_1\setminus\{P\}$, $Q_2 \in \mathcal{P}_2\setminus\{P\}$ and define $\mathcal{P}_1' = \mathcal{P} \setminus \{Q_1\}$ and $\mathcal{P}_2' = \mathcal{P} \setminus \{Q_2\}$.
The vertex set $V$ of our hypergraph $\mathcal{H}$ is equal to $\mathcal{P}_1' \cup \mathcal{P}_2' \cup \{v\}$ where $v$ is a vertex that is not a point of any of the planes.
Let $\mathcal{C}$ be a conic of $\pi_2$ through $Q_2$ such that the line $PQ_2$ is tangent to $\mathcal{C}$ and let $\mathcal{C}'=\mathcal{C}\setminus \{Q_2\}$. The edge set of $\mathcal{H}_1$ is the union of the following sets:

\begin{itemize}
    \item the set $\mathcal{E}_1$ of all the lines of $\pi_1$ that do not pass through $Q_1$ or $P$ and a line $\ell$ through $P$ that does not contain $Q_1$;

    \item the set $\mathcal{E}_2$ of the lines of $\pi_2$ that do not pass through $Q_2$ and intersect $\mathcal{C}'$ non-trivially;

    \item the set $\{e_1, e_2\}$, where $e_1 = \ell \setminus \{P\} \cup\{R\}$, with $R \in PQ_2\setminus\{P,Q_2\}$, and $e_2 = \mathcal{C}' \cup \{v\}$.
\end{itemize}

 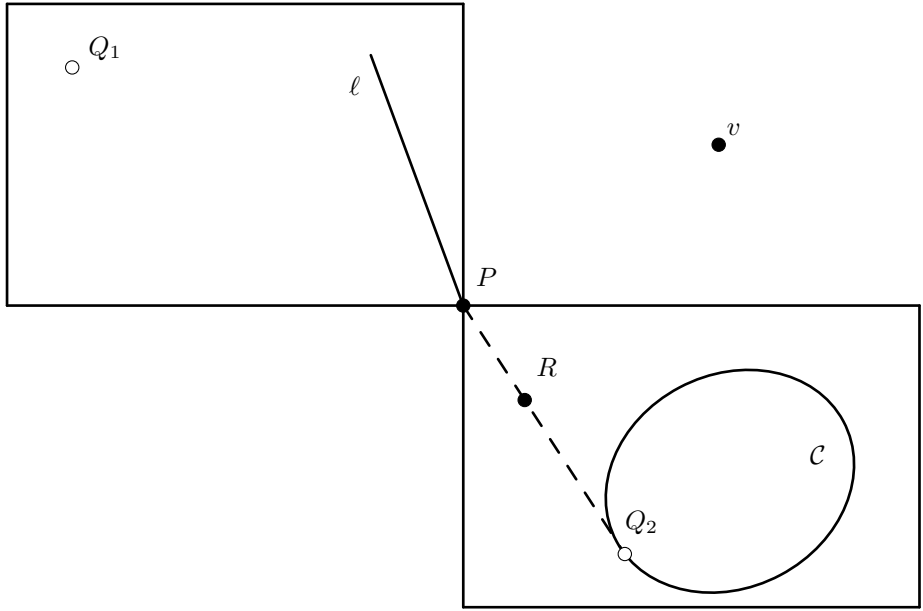
\begin{figure}[ht]
 \label{fig:first}
    \centering
    \definecolor{uuuuuu}{rgb}{0.0,0.0,0.0}
\begin{tikzpicture}[line cap=round,line join=round,>=triangle 45,x=1.0cm,y=1.0cm]
 \fill[line width=1.pt,color=white] (0.,0.) -- (0.,4.) -- (-6.,4.) -- (-6.,0.) -- cycle;
\fill[line width=1.pt,color=white] (0.,0.) -- (6.,0.) -- (6.,-4.) -- (0.,-4.) -- cycle;
\draw [line width=1.pt,color=uuuuuu] (0.,0.)-- (0.,4.);
\draw [line width=1.pt,color=uuuuuu] (0.,4.)-- (-6.,4.);
\draw [line width=1.pt,color=uuuuuu] (-6.,4.)-- (-6.,0.);
\draw [line width=1.pt,color=uuuuuu] (-6.,0.)-- (0.,0.);
\draw [line width=1.pt,color=uuuuuu] (0.,0.)-- (6.,0.);
\draw [line width=1.pt,color=uuuuuu] (6.,0.)-- (6.,-4.);
\draw [line width=1.pt,color=uuuuuu] (6.,-4.)-- (0.,-4.);
\draw [line width=1.pt,color=uuuuuu] (0.,-4.)-- (0.,0.);
\draw [line width=1.pt,dash pattern=on 6pt off 6pt] (0.,0.)-- (2.124617600983123,-3.294700575208272);
\draw [line width=1.pt] (0.,0.)-- (-1.2159840761855436,3.3206341373963424);
\draw [rotate around={-151.85566679100694:(3.5069707917345965,-2.3300019033464703)},line width=1.pt] (3.5069707917345965,-2.3300019033464703) ellipse (1.6908161911538424cm and 1.4094678957135336cm);

\draw [fill=uuuuuu] (0.,0.) circle (2.5pt);
\draw[color=uuuuuu] (0.30772989822551783,0.3740149784368928) node {$P$};
\draw [color=uuuuuu] (-5.142933210047298,3.1575801127898906) circle (2.5pt);
\draw[color=uuuuuu] (-4.7120047164445324,3.413807865742886) node {$Q_1$};
\draw[color=black] (-1.4256249649652674,2.919825347533669) node {$\ell$};
\draw[color=black] (4.663601698426444,-1.9786216623133397) node {$\mathcal{C}$};
\draw [fill=black] (0.8078859073372373,-1.2528099938431208) circle (2.5pt);
\draw[color=black] (1.0997065891711406,-0.8139500579815414) node {$R$};
 \draw [fill=uuuuuu] (3.35917, 2.13267) circle (2.5pt);
 \draw[color=uuuuuu] (3.55917, 2.33267) node {$v$};
 \draw [fill=white] (2.124617600983123,-3.294700575208272) circle (2.5pt);
\draw[color=uuuuuu] (2.3459052058061647,-2.8754187976488246) node {$Q_2$};

\end{tikzpicture}
\caption{First construction}
    \end{figure}

In the following lemmas, we will prove that the hypergraph $\mathcal{H}_1$ is a $(q + 1)$-Ryser hypergraph with $\nu=2$, when $q$ is an odd prime.
We only need to assume that $q$ is an odd prime when proving $\tau(\mathcal{H}_1) = 2q$.

\begin{lem}\label{lem:partite2}
$\mathcal{H}_1$ is a $(q + 1)$-partite hypergraph.
\end{lem}
\begin{proof}
Let $\ell_1, \dots, \ell_{q + 1}$ be the lines through $Q_1$ in $\pi_1$ and let $m_1, \dots m_{q + 1}$ be the lines through $Q_2$ in $\pi_2$ such that $\ell_1 = Q_1P$ and $m_1 = Q_2P$.
Let $V_1 = (\ell_1 \setminus \{Q_1\}) \cup (m_1 \setminus \{Q_2\}) \cup \{v\}$ and $V_i = (\ell_i \setminus \{Q_1\}) \cup (m_i \setminus \{Q_2\})$ for  $i = 2, \dots, q + 1$.
Then these $q + 1$ form the sides of the hypergraph.
Indeed, an edge of $\mathcal{E}_1$ or $\mathcal{E}_2$ intersects $V_i$ in exactly one point since this edge is a line of one of the projective planes. The edge $e_1$ also intersects each $V_i$ in one point since $e_1 \cap V_i=\ell \cap V_i$ for $i \neq 1$ and  $e_1 \cap V_1=\{R\}$. For all $i > 1$, $m_i$ intersects $\mathcal{C}'$ in exactly one point, and $v$ is contained in $V_1$; hence $e_2 = \mathcal{C}' \cup \{v\}$ intersects $V_i$ in one point for all $i = 1, \dots, q + 1$.
\end{proof}

\begin{lem}\label{lem:matching2}
The matching number of $\mathcal{H}_1$ is equal to $2$.
\end{lem}
\begin{proof}
A matching of size $2$ can be obtained by taking any $f_1 \in \mathcal{E}_1 \setminus \{\ell\}$ and $f_2 \in \mathcal{E}_2$.
The edge set of $\mathcal{H}_1$ can be partitioned into two sets, $\mathcal{E}_1 \cup \{e_1\}$ and $\mathcal{E}_2 \cup \{e_2\}$, such that the edges in each of these sets pairwise intersect each other non-trivially.
Therefore, by the pigeonhole principle any matching has size at most $2$.
\end{proof}


\begin{lem}\label{lem:cover2}
If $q$ is an odd prime, then the vertex cover number of $\mathcal{H}_1$ is equal to $2q$.
\end{lem}
\begin{proof}
Any side is a vertex cover (blocking set) of size $2q$.
Let $B$ be a minimal vertex cover of $\mathcal{H}_1$ and let $B_1 = B \cap (\mathcal{P}_1' \setminus \{P\})$ and $B_2 = B \cap \mathcal{P}_2'$. Then $B_1$ must block all the edges of $\mathcal{E}_1\setminus \{\ell\}$ since these edges do not contain any point of $\mathcal{P}_2' \cup \{v\}$, and hence $B_1 \cup \{P,Q_1\}$ is a blocking set for $\pi_1$.
By Lemma \ref{lem:blocking_basis}, this means that $|B_1|\geq q-1$ and $|B_1|=q-1$ if and only if $B_1= PQ_1 \setminus\{P,Q_1\}$.
The set $B_2$ is a blocking set for the edges of $\mathcal{E}_2$, but we have observed that $(\mathcal{P}_2',\mathcal{E}_2)$ is an intersecting Ryser hypergraph, and hence $|B_2|\geq q$.
Thus, $|B| \geq |B_1| + |B_2| \geq 2q - 1$.

Suppose that $|B_1 \cup B_2|=2q-1$.
Then we must have $B_1=PQ_1 \setminus\{P,Q_1\}$, and hence $\ell$ must be blocked by $B_2$; so $P \in B_2$.
Also, the edge $e_1$ cannot be blocked by $B_1$, and hence $R \in B_2$.
If $v \in B_2$, then $B_2\setminus\{v\}$ is a vertex cover for $(\mathcal{P}_2',\mathcal{E}_2)$, which means that $|B_2|\geq q+1$; a contradiction.
Hence $B_2$ intersects $e_2$ in a point of $\mathcal{C}'$.
The set $S=B_2 \cup \{Q_2\}$ is a blocking set for the non-external lines to $\mathcal{C}$. The set $S$ does not coincide with $\mathcal{C}$ since it contains $P$ and $R$.
Also, $S$ is not a line, since it contains two points of $PQ_2$ and a point of $\mathcal{C}'$, which cannot be in $PQ_2$ as $PQ_2 \cap \mathcal{C}=\{Q_2\}$.
Hence, by Proposition \ref{blockconic}, the points of $B_2$ not in the conic are in a unique line.
Since $S$ contains two points of the line $PQ_2$ that are not in $\mathcal{C}$, the points of $S$ not in $\mathcal{C}$ are all in $PQ_2$ and $|S \cap PQ_2|$ must be the size of a subgroup of $G$, where $G$ is an elementary abelian group of order $q$.
Since $q$ is an odd prime it has no non-trivial subgroups, and since $|S\cap PQ_2| \geq 2$, we get a contradiction.
\end{proof}

Finally, we prove that the vertex set of $\mathcal{H}_1$ cannot be partitioned into two sets $V_1, V_2$ such that there are subhypergraphs of $\mathcal{H}_1$ on $V_1$ and $V_2$ which are both intersecting Ryser hypergraphs.

\begin{lem}\label{lem:partition2}
If $q$ is odd, then $\mathcal{H}_1$ cannot contain two vertex disjoint intersecting $(q + 1)$-Ryser hypergraphs.
\end{lem}

\begin{proof}
Suppose that $(V_1,\mathcal{F}_1)$ and $(V_2,\mathcal{F}_2)$ are two intersecting $(q+1)$-Ryser hypergraphs contained in $\mathcal{H}_1$ and $V_1 \cap V_2=\emptyset$.
We will assume that these subhypergraphs have no isolated vertices, since one can remove such vertices while maintaining the property of being a Ryser hypergraph.
Every edge of $\mathcal{H}_1$ contains at least $q$ points of the same projective plane.
If $\mathcal{F}_i$ contains two edges, say $f_1$ and $f_2$, such that $|f_1 \cap \mathcal{P}'_1|\geq q$ and $|f_2 \cap \mathcal{P}'_2| \geq q$, then $f_1\cap f_2$ is either $P$ or $R$, but then all the edges of $\mathcal{F}_i$ must contain $P$ or $R$ respectively, so the cover number of $(V_i,\mathcal{F}_i)$ is $1$, a contradiction.
So without any loss of generality, we can assume that $|f\cap \mathcal{P}'_1| \geq q$ $\forall f \in \mathcal{F}_1$ and $|f\cap \mathcal{P}'_2| \geq q$ $\forall f \in \mathcal{F}_2$.
This implies that $\mathcal{F}_1 \subseteq \mathcal{E}_1 \cup \{e_1\}$ and $\mathcal{F}_2 \subseteq \mathcal{E}_2 \cup \{e_2\}$.
Since every vertex of $V_i$ must be incident to at least one edge of $\mathcal{F}_i$, we also get that $V_1 \subseteq \mathcal{P}'_1 \cup \{R\}$ and $V_2 \subseteq \mathcal{P}'_2 \cup \{v\}$.
A side in every intersecting $(q+1)$-Ryser hypergraph must have size at least $q$ (otherwise the vertex cover number will be less than $q$), therefore, we must have $|V_i|\geq q^2+q$ for $i=1,2$.
If $P \notin V_1$, then $V_1$ must contain $R$, and we get $e_1\in \mathcal{F}_1$. Since $V_1 \cap V_2=\emptyset$, $R \notin V_2$, and hence $\mathcal{F}_2$ does not contain any edge containing $R$. As $q$ is odd, there are exactly two tangent lines to the conic $\mathcal{C}$ through $R$, one of them being $RQ_2$.
Let $RR'$ be the other tangent, with $R' \in \mathcal{C}'$.
The set $\mathcal{C}' \setminus \{R'\}$ intersects every edge in $\mathcal{F}_2$ non-trivially because the only edge it can possibly miss is $RR'$, but since $R$ is not in $V_2$, this edge is not contained in $\mathcal{F}_2$.
This gives us a vertex cover of size $q - 1$, which implies that $(V_2, \mathcal{F}_2)$ is not a $(q + 1)$-Ryser hypergraph.
So we must have $P \in V_1$ and $P \notin V_2$, hence $\mathcal{F}_2$ cannot contain any edge through $P$. Then we again get a vertex cover of size $q - 1$ in $(V_2, \mathcal{F}_2)$  by taking $\mathcal{C}' \setminus \{P'\}$ where $P'$ is the unique point on $\mathcal{C}$ other than $Q_2$ for which $PP'$ is a tangent.
Thus, we have a contradiction.
\end{proof}

Once we have the hypergraph $\mathcal{H}_1$ as above, then for any $\nu > 2$ we can take $\nu - 2$ disjoint copies of truncated projective planes, and a disjoint copy of $\mathcal{H}_1$ to get a Ryser hypergraph with matching number $\nu$ whose vertex set cannot be partitioned into $\nu$ intersecting Ryser hypergraphs.
We now prove something stronger.

\begin{thm}
Let $r-1$ be an odd prime and $\nu > 1$. Then there exits an $r$-Ryser hypergraph $\mathcal{H}$ with $\nu(\mathcal{H})= \nu$ that does not contain two vertex disjoint intersecting $r$-Ryser hypergraphs.
\end{thm}
\begin{proof}
Let $\pi_i,i=1,2,\ldots,\nu$ be projective planes over the finite field of order $q=r-1$, all sharing a common point $P$.
Let $Q_i \in \pi_i$ a point $\neq P$.
The vertices of $\mathcal{H}$ are the points of $\pi_i \setminus\{Q_i\},i=1,2,\ldots,\nu$ and $\nu-1$ more vertices, namely $v_2,v_2,\ldots,v_{\nu}$.
Let $\mathcal{C}_i,i=2,3,\ldots \nu$, be a conic of $\pi_i$ such that the line $PQ_i$ is tangent to $\mathcal{C}_i$ in $Q_i$.
The edges containing only points of $\pi_1$ are constructed in the same way as in $\mathcal{E}_1$, the edges containing only points of $\pi_i,i=2,3,\ldots,\nu$ are constructed in the same way as in $\mathcal{E}_2$.
Let $R_i$ be a point of $PQ_i$, $,i=2,3,\ldots \nu$, not equal to $P$ or $Q_i$.
Then we also have the edges $e_1^{(i)}:=\ell \setminus\{P\}\cup \{R_i\}$ , $e_2^{(i)}:= \mathcal{C}_i' \cup v_i$, $,i=2,3,\ldots \nu$.
The proofs and statements of all the previous lemmas can easily be adapted to this hypergraph; observe that in the proof of Lemma \ref{lem:partition2} we only used disjointness of the vertex sets and not that they form a partition.
\end{proof}

\section{Second construction}\label{sec:first}
Let $\pi_1, \pi_2$ be two Desarguesian projective planes of order $q \geq 4$ with point sets $\mathcal{P}_1$ and $\mathcal{P}_2$, respectively, such that $\mathcal{P}_1 \cap \mathcal{P}_2$ is a point, say $\{P\}$.
Let $Q_1 \in \mathcal{P}_1\setminus\{P\}$ and $Q_2 \in \mathcal{P}_2\setminus\{P\}$.
Define $\mathcal{P}_1' = \mathcal{P} \setminus \{Q_1\}$ and $\mathcal{P}_2' = \mathcal{P} \setminus \{Q_2\}$.
The vertex set $V$ of our hypergraph $\mathcal{H}_2$ is equal to $\mathcal{P}_1' \cup \mathcal{P}_2' \cup \{v\}$ where $v$ is a new vertex that is contained neither in $\mathcal{P}_1$ nor $\mathcal{P}_2$.

Let $\ell$ be a line through $P$ but not through $Q_1$ in $\pi_1$.
Let $T_1, T_2, T_3 \in \mathcal{P}_2 \setminus \{Q_2, P\}$ such that $\{T_1, T_2, T_3, Q_2\}$ is an arc and $T_1$ lies on the line $PQ_2$.
Let $S$ be a point on the line $Q_2 T_2$, distinct from $Q_2$ and $T_2$ \footnote{For $q = 4$ we will also assume that $P$ and $S$ are not contained in the  unique Baer subplane containing $Q_2, T_1, T_2$ and $T_3$, and that $S$ is not contained on the line $T_1T_3$; there is a unique such choice for $S$.}.
The edge set $\mathcal{E}$ of the hypergraph $\mathcal{H}_2$ is the union of the following three sets:

\begin{itemize}
    \item the set $\mathcal{E}_1$ of all the lines of $\pi_1$ that do not pass through $Q_1$ or $P$ and the line $\ell$;

    \item the set $\mathcal{E}_2$ of all the lines of $\pi_2$ that do not contain any point of the set $\{T_1,T_2,T_3,Q_2\}$, and the three lines $T_1T_2,T_1S$ and $PT_3$;

    \item the set $\{e_1, e_2\}$ where $e_1 = \ell \setminus \{P\} \cup\{ T_1\}$ and $e_2 = T_1T_2 \setminus \{T_1, T_2\} \cup \{v, S\}$.
\end{itemize}

    \begin{figure}[ht]
    \centering
    \definecolor{uuuuuu}{rgb}{0.0,0.0,0.0}
\begin{tikzpicture}[line cap=round,line join=round,>=triangle 45,x=1.0cm,y=1.0cm]
\fill[line width=1.pt, fill = white] (0.,0.) -- (0.,4.) -- (-6.,4.) -- (-6.,0.) -- cycle;
\fill[line width=1.pt, fill = white] (0.,0.) -- (6.,0.) -- (6.,-4.) -- (0.,-4.) -- cycle;
\draw [line width=1.pt,color=black] (0.,0.)-- (0.,4.);
\draw [line width=1.pt,color=black] (0.,4.)-- (-6.,4.);
\draw [line width=1.pt,color=black] (-6.,4.)-- (-6.,0.);
\draw [line width=1.pt,color=black] (-6.,0.)-- (0.,0.);
\draw [line width=1.pt,color=black] (0.,0.)-- (6.,0.);
\draw [line width=1.pt,color=black] (6.,0.)-- (6.,-4.);
\draw [line width=1.pt,color=black] (6.,-4.)-- (0.,-4.);
\draw [line width=1.pt,color=black] (0.,-4.)-- (0.,0.);
\draw [line width=1.pt,dash pattern=on 6pt off 6pt] (0.,0.)-- (5.478871821458702,-3.6208086244211755);
\draw [line width=1.pt] (0.,0.)-- (-2.0449067425247147,3.483688162002794);
\draw [line width=1.pt,dash pattern=on 6pt off 6pt] (0.796891972044873,-3.1782334147750917)-- (5.851566734844877,-3.2015268468617277);
\draw [line width=1.pt] (0.,0.)-- (5.129470340159163,-1.780627489576934);
\draw [line width=1.pt] (2.018676571094474,-1.3340778497332715)-- (3.0563548844485617,-3.713982352767719);
\draw [line width=1.pt] (3.0563548844485617,-3.713982352767719)-- (1.542281798817224,-0.21996753977232433);
\draw [line width=1.pt] (1.1695868854310485,-3.154939982688456)-- (2.4041387860227545,-0.47619529272531996);
\draw [line width=1.pt] (2.4041387860227545,-0.47619529272531996)-- (0.9366525645646889,-3.6208086244211755);

\draw [fill=uuuuuu] (0.,0.) circle (2.5pt);
\draw[color=uuuuuu] (0.30772989822551783,0.3740149784368928) node {$P$};
\draw[color=uuuuuu] (5.047943327855937,-2.758951637215645) node {$Q_2$};
\draw [color=uuuuuu] (-5.142933210047298,3.1575801127898906) circle (2.5pt);
\draw [color=uuuuuu] (4.82666, -3.17823) circle (2.5pt);
\draw[color=uuuuuu] (-4.7120047164445324,3.413807865742886) node {$Q_1$};
\draw [fill=uuuuuu] (2.018676571094474,-1.3340778497332715) circle (2.5pt);
\draw[color=uuuuuu] (2.5322526624992525,-1.3147588478442147) node {$T_1$};
\draw [fill=uuuuuu] (3.9880921679140005,-1.3846391441041226) circle (2.5pt);
\draw[color=uuuuuu] (4.39572722943013,-1.151704823237763) node {$T_3$};
\draw [fill=uuuuuu] (1.1695868854310485,-3.154939982688456) circle (2.5pt);
\draw[color=uuuuuu] (0.855125552261463,-2.8288319334755525) node {$T_2$};
\draw [fill=uuuuuu] (2.823270187634987,-3.1875715632340325) circle (2.5pt);
\draw[color=uuuuuu] (3.102941748621834,-2.9103589457787784) node {$S$};
\draw [fill=uuuuuu] (3.35917, 2.13267) circle (2.5pt);
\draw[color=uuuuuu] (3.55917, 2.33267) node {$v$};

\draw[color=uuuuuu] (-2.34491, 3.08369) node {$\ell$};

\end{tikzpicture}
\caption{Second construction}
    \end{figure}
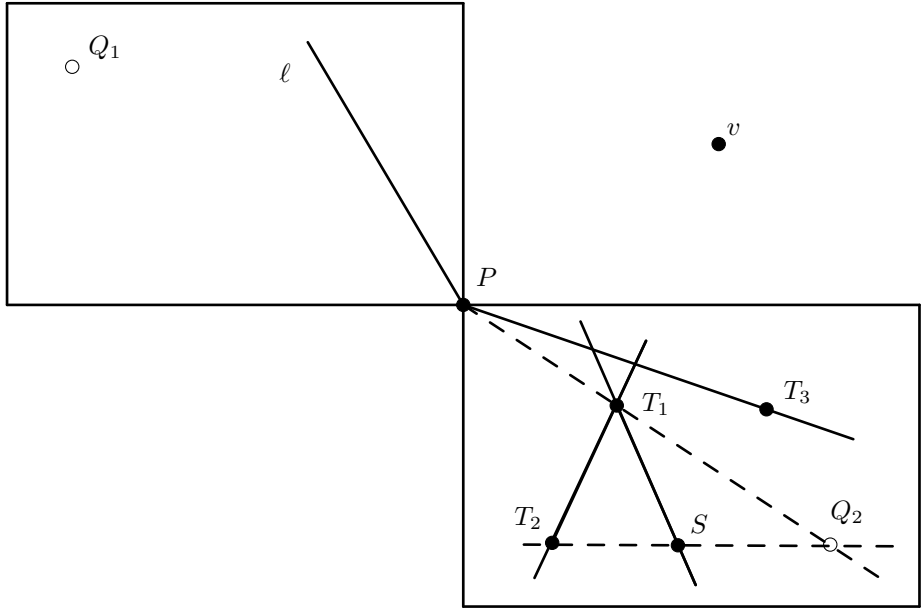

\begin{lem}\label{lem:partite}
$\mathcal{H}_2$ is a  $(q + 1)$-partite hypergraph.
\end{lem}
\begin{proof}
Let $\ell_1, \dots, \ell_{q + 1}$ be the lines through $Q_1$ in $\pi_1$ and let $m_1, \dots m_{q + 1}$ be the lines through $Q_2$ in $\pi_2$ such that $\ell_1 = Q_1P$ and $m_1 = Q_2P$.
Let $V_1 = (\ell_1 \setminus \{Q_1\}) \cup (m_1 \setminus \{Q_2\}) \cup \{v\}$ and $V_i = (\ell_i \setminus \{Q_1\}) \cup (m_i \setminus \{Q_2\})$ for all $i > 1$.
Then these $q + 1$ sets form a partition of the vertex set of $\mathcal{H}_2$.
If $e$ is an edge of $\mathcal{H}_2$ which is a line of $\pi_1$ or $\pi_2$, then it intersects each $V_i$ in the unique point where it intersects the line $\ell_i$ or $m_i$.
The edge $e_1$ intersects all $V_i$ for $i > 1$ in a unique point, $\ell \cap V_i$, and its unique point in $V_1$ is $T_1$.
The edge $e_2$ intersects $V_1$ in $v$ and the side corresponding to the line $Q_2T_2$ in $S$; it intersects every other side in the unique point where the line $T_1T_2$ intersects the side.
\end{proof}


\begin{lem}\label{lem:matching}
The matching number of $\mathcal{H}_2$ is equal to $2$.
\end{lem}
\begin{proof}
The proof is the same as in the first construction.
\end{proof}

To be able to determine $\tau(\mathcal{H}_2)$ we need the folllowing results on blocking sets in finite projective planes.
Besides lines,  an important class of blocking sets of the plane are the Baer subplanes. If in PG$(2,q^2)$  we  only consider  the vectors with coordinates in $\mathbb{F}_q$, then we obtain a copy of PG$(2,q)$. Any substructure of PG$(2,q^2)$ equivalent to that by the action of the projective linear group, is called \textit{Baer subplane} of PG$(2,q^2)$.
We stress out that a Baer subplane of PG$(2,q^2)$ is uniquely determined by a $4$ points in general position, that is, an arc of size $4$.

\begin{prop}[Bruen-Thas \cite{B70,BT77}]\label{prop:Bruen-Thas}
Let $B$ be a blocking set in a projective plane of order $q$. If $B$ does not contain any line of the plane then $|B| \geq q + \sqrt{q} + 1$ and the equality is obtained if and only if $q$ is a square and $B$ is a Baer subplane.
\end{prop}

\begin{prop}[Blokhuis \cite{B94}]\label{prop:Blokhuis}
Let $B$ be a blocking set in $\mathrm{PG}(2,p)$, with $p$ an odd prime, that does not contain any line. Then $|B| \geq 3(p+1)/2$.
\end{prop}

\begin{prop}[Blokhuis \cite{B96}]\label{prop:Blokhuiscube}
Let $B$ be a blocking set in $\mathrm{PG}(2,p^3)$, with $p$ a prime, that does not contain any line. Then $|B| \geq p^3 + p^2+1$.
\end{prop}


\begin{lem}\label{lem:cover}
The vertex cover number of $\mathcal{H}_2$ is equal to $2q$.
\end{lem}
\begin{proof}
Any side is a vertex cover of size $2q$.
Now let $B$ be a minimal vertex cover of $\mathcal{H}_2$.
Let $B_1 = B \cap (\mathcal{P}_1' \setminus \{P\})$ and $B_2 = B \cap \mathcal{P}_2'$.
Note that $\mathcal{E}_1 \setminus \{\ell\}$ must be blocked by $B_1$, as the edges in this set do not contain any point of $B_2 \cup \{v\}$. Hence $B_1\cup \{P,Q_1\}$ is a blocking set for $\pi_1$, by which we get $|B_1| \geq q - 1$ for all $q$, with equality if and only if $B_1 = PQ_1 \setminus \{P, Q_1\}$ by Lemma \ref{lem:blocking_basis}.

Let $A = B_2 \cup \{Q_2, T_1, T_2, T_3\}$.
Since $B_2$ blocks all elements of $\mathcal{E}_2$, $A$ forms a blocking set of the lines of the plane $\pi_2$.
Say  $A$ contains a line $m$ of the plane $\pi_2$.
Note that $m$ intersects the set $\{Q_2, T_1, T_2, T_3\}$ in at most two points.
If $m$ is disjoint from the set, then $B_2$ contains $q + 1$ points (of the line $m$), and hence $|B| \geq |B_1|+ |B_2| \geq q - 1 + q + 1 = 2q$.
Say $m$ intersects $\{Q_2, T_1, T_2, T_3\}$ in exactly one point.
If $T_1 \not \in B_2$, then the edge $e_1 = \ell \setminus \{P\} \cup \{T_1\}$ must be blocked by a point of $B_1$ and we get $|B_1| \geq q$; and since $m \setminus \{Q_2, T_1, T_2, T_3\}$ is a subset of $B_2$, of cardinality $q$, we get $|B| \geq 2q$.
If $T_1 \in B_2$, then either $T_1 \in m$ and $m \subseteq B_2$, or $T_1 \notin m$ and $B_2$ contains $T_1$ and $q$ points of $m$, hence in both cases $|B_2| \geq q + 1$ and we again get $|B| \geq 2q$.
Finally assume that $|m \cap \{Q_2, T_1, T_2, T_3\}| = 2$. We have $B_2\cap m \geq q-1$. Suppose that $|B_2|=q-1$, then $B_2=m\setminus\{ m \cap \{Q_2, T_1, T_2, T_3\}\}$. Hence $B_2$ cannot block an edge intersecting $m$ in a point of $m \cap \{Q_2, T_1, T_2, T_3\}$. We remind that $\{T_1,T_2,T_3,Q_2\}$ is an arc of $\pi_2$, i.e., no three of those points are in a line. If $m=T_1T_2$, the edge $T_1S$ is not blocked (if $S\in T_1T_2$, then $T_1T_2$ would contain the point $Q_2$, a contradiction); if $m$ is $T_1T_3$ or $Q_2T_1$ or $Q_2T_2$ or $T_2T_3$, then the edge $T_1T_2$ is not blocked; if $m=Q_2T_3$, then the edge $PT_3$ is not blocked (if $P \in Q_2T_3$, then $Q_2T_3$ would contain the point $T_1$, a contradiction). Hence
we have $|B_2| \geq q$.
If $|B_1| \geq q$ or $|B_2| \geq q + 1$, then we are done.
So, say $|B_1| = q - 1$ and $|B_2| = q$, then $B_1 = PQ_1 \setminus \{P, Q_1\}$ and $B_2$ has at most one point not in $m$.
Since $\ell$ needs to be blocked, we must have $P \in B_2$.
Since we need to block the edge $e_1$, the point $T_1$ must be in $B_2$. Suppose that $P \in m$, then $m$ can be either $T_2T_3$ or $T_1Q_2$. If $m=T_2T_3$, then $T_1 \notin m$ and hence $B_2=T_2T_3\setminus\{T_2,T_3\}\cup T_1$. If $S \in T_2T_3$, then $T_2,T_3,Q_2$ would be in the same line, a contradiction. But then  $e_2 = T_1T_2 \setminus \{T_1, T_2\} \cup \{S, v\}$ is not blocked by any point of $B_1 \cup B_2$, and hence $|B| \geq 2q$. If $m=T_1Q_2$, then $B_2=T_1Q_2\setminus\{Q_2\}$. Again $S\notin B_2$, hence $e_2$ is not blocked  by any point of $B_1 \cup B_2$ and $|B| \geq 2q$. Finally, suppose that $T_1\in B_2$, hence $B_2$ consists of $q$ points of $m$, hence $P\in m$ and again $m=T_1Q_2$ and $B_2=T_1Q_2\setminus\{Q_2\}$.
 
But then the edge $e_2 = T_1T_2 \setminus \{T_1, T_2\} \cup \{S, v\}$ is not blocked by any point of $B_1 \cup B_2$, and hence $|B| \geq 2q$.

Now assume that $A$ does not contain any line of $\pi_2$.
If $T_1 \not \in B_2$, then as before the edge $e_1$ must be blocked by a point of $B_1$.
This implies that $|B_1| \geq q$, as otherwise $B_1$ does not contain any point of $\ell \setminus \{P\}$.
If $T_1 \in B_2$, then we have $|B_2| \geq |A| - 3$.
In both cases, we get $|B| \geq |B_1| + |B_2| \geq q + |A| - 4$.
The lower bounds on $|A|$ given by Propositions \ref{prop:Bruen-Thas}, \ref{prop:Blokhuis} and \ref{prop:Blokhuiscube} then imply that $|B| \geq 2q$ for all prime powers $q \geq 5$.

Finally assume that $q = 4$ and $A$ does not contain any line of $\pi_2$.
We have $q+|A|-4 \geq q+q+\sqrt{q}+1-4=2q-1$, so $|B|$ can only be less then $2q$ only if $|A|=q+\sqrt{q}+1$, that is, $A$ is a Baer subplane and either $|B_1|=q-1$ and $T_1 \in B_2$ or $|B_1|=q$ and $B_2 \cap  \{Q_2, T_1, T_2, T_3\}=\emptyset$.
In the first case, $B_1= PQ_1 \setminus \{P, Q_1\}$, but we have assumed that the point $P$ is not contained in the unique Baer subplane containing $Q_2, T_1, T_2, T_3$, that is, $P \notin A$, and thus the line $\ell$ is not blocked by $B_1\cup B_2$.
In the second case, $B_2$ is equal to the Baer subplane minus the set $\{Q_2, T_1, T_2, T_3\}$.
The three lines through $T_1$\footnote{Note that this is a subplane of order $2$} inside the Baer subplane correspond to $T_1T_2, T_1T_3$ and $T_1Q_2$.
Since we made sure that $S$ is not in the Baer subplane, and not on the line $T_1T_3$, we know that the edge $T_1S$ of the hypergraph $\mathcal{H}_2$ contains a unique point of the Baer subplane, namely $T_1$, and hence it is not blocked by $B_2$, which is a contradiction.

\end{proof}

\begin{lem}\label{lem:partition}
$\mathcal{H}_2$ cannot contain two disjoint intersecting $(q + 1)$-Ryser hypergraphs.
\end{lem}
\begin{proof}
Suppose that $(V_1,\mathcal{F}_1)$ and $(V_2,\mathcal{F}_2)$ are two intersecting $(q+1)$-Ryser hypergraphs contained in $\mathcal{H}_2$ and $V_1 \cap V_2=\emptyset$. As in Lemma \ref{lem:partition2}, we will assume that these subhypergraphs have not isolated vertices.
If $\mathcal{F}_i$ contains $f_1 \in \mathcal{E}_1\cup \{e_1\}$ and $f_2 \in \mathcal{E}_2\cup \{e_2\}$, then, in order to have $(V_i,\mathcal{F}_i)$ intersecting, $f_1\cap f_2$ is either $P$ or $T_1$ and all the edges of $\mathcal{F}_i$ must contain $P$ or $T_1$, giving us a cover number $1$, a contradiction.
Hence, we can assume $\mathcal{F}_1\subseteq  \mathcal{E}_1\cup \{e_1\}$ and  $\mathcal{F}_2\subseteq  \mathcal{E}_2\cup \{e_2\}$ and therefore $V_1\subseteq  \mathcal{P}_1'\cup \{T_1\}$ and $V_2 \subseteq \mathcal{P}_2' \cup \{v\}$.
A side in every intersecting $(q+1)$-Ryser hypergraph must have size at least $q$ (otherwise the vertex cover number will be less than $q$), therefore, we must have $|V_i|\geq q^2+q$ for $i=1,2$.
Suppose that $P \notin V_1$, then $T_1 \in V_1$ and $\mathcal{F}_2$ does not contain any edge through $T_1$.
The only edge containing $T_2$ is the line $T_1T_2$, hence $T_2 \notin V_2$ and $V_2 \subseteq \mathcal{P}_2 \setminus\{T_1,T_2\}\cup \{v\}$, which is a subset of size $q^2 + q - 1$, a contradiction.
So let $P \in V_1$, then $\mathcal{F}_2$ cannot contain any edge with $P$, but the only edge containing $T_3$ is the line $PT_3$, so $T_3 \notin V_2$ and $V_2 \subseteq \mathcal{P}_2 \setminus\{P,T_3\}\cup \{v\}$, a contradiction again.
\end{proof}

\begin{thm}
Let $r-1 \geq 4$ be a prime power and $\nu > 1$. Then there exits an $r$-Ryser hypergraph $\mathcal{H}$ with $\nu(\mathcal{H}) = \nu$ that does not contain two vertex disjoint intersecting $r$-Ryser hypergraphs.
\end{thm}
\begin{proof}
Let $\Pi_i,i=1,2,\ldots,\nu$ be projective planes over the finite field of order $q=r-1$ sharing a point $P$.
Let $Q_i \in \Pi_i$ a point $\neq P$.
The vertices of $\mathcal{H}$ are the points of $\Pi_i \setminus\{Q_i\},i=1,2,\ldots,\nu$ with $\nu-1$ more vertices, namely $v_2,v_3,\ldots,v_{\nu}$. The edges containing only points of $\pi_1$ are constructed in the same way as in $\mathcal{E}_1$, the edges containing only points of $\Pi_i,i=2,3,\ldots,\nu$ are constructed in the same way as in $\mathcal{E}_2$, for which we let $T_1^{(i)},T_2^{(i)},S^{(i)},i=2,3,\ldots,\nu$ be the points of $\pi_i$ playing the same role as $T_1,T_2,S$ in $\pi_2$. Then we also have the edges $e_1^{(i)}:=\ell \setminus\{P\}\cup T_1^{(i)}$ , $e_2^{(i)}:=T_1^{(i)}T_2^{(i)}\setminus\{T_1^{(i)},T_2^{(i)}\}\cup \{v_i,S^{(i)}\}$, $i=2,3,\ldots,\nu$.
The proof of all of the previous lemmas can now be easily adapted to this hypergraph.
\end{proof}

\section{Concluding remarks}
Our constructions indicate that the non-intersecting Ryser hypergraphs have a richer structure than the intersecting Ryser hypergraphs, and it might be fruitful  to look at the non-intersecting case for new constructions of Ryser hypergraphs, or to disprove Ryser's conjecture.
We have shown that there are non-intersecting Ryser hypergraphs (with arbitrary matching number) that do not contain two disjoint sets of vertices on which we have intersecting Ryser subypergraphs.
But it is not at all clear whether one should always have even a single intersecting Ryser subhypergraph.

\medskip
\noindent
\textbf{Problem 1}: {\em Does every non-intersecting $r$-Ryser hypergraph contain a subhypergraph that is an intersecting $r$-Ryser hypergraph?}

\medskip
\noindent
Note that a negative answer to this question will disprove the following stronger version of Ryser's conjecture due to Lov\'{a}sz \cite{Lovasz}:
\begin{conj}
In every $r$-partite hypergraph $\mathcal{H}$, there exists a set $S$ of vertices of size at most $r - 1$ such that $\nu(H - S) \leq \nu(H) - 1$.
\end{conj}
\noindent
We would also like to note that our construction does not contradict the ``vertex-minimal analogue'' of the characterization of $3$-Ryser hypergraphs proposed by Abu-Khazneh and Pokrovskiy \cite[Ch.~6]{A16}.
It will be interesting to find counterexamples to that conjecture.

\medskip
\noindent
In the theory of finite projective planes, a blocking set is called \textit{non-trivial} if it does not contain any line.
Inspired by this concept, we ask the following question.

\medskip
\noindent
\textbf{Problem 2}:
{\em In an intersecting $r$-partite hypergraph, what is the smallest size of a vertex cover that does not contain any edge or side?}

\medskip
\noindent
Note that the smallest size of a non-trivial blocking set in $\mathrm{PG}(2,p)$, where $p$ is a prime, is equal to $3(p + 1)/2$ \cite{B94}.
Truncating these planes gives rise to an intersecting $(p + 1)$-partite hypergraph in which the smallest non-trivial vertex cover has size $(3p + 1)/2$.
Therefore, in a general upper bound for \textbf{Problem 2}, we cannnot do better than $1.5r + o(r)$.

\section*{Acknowledgements}
We would like to thank the referee for their very thorough reading and helpful comments.

\appendix
\section{The case of $r = 4$}
In this appendix we show that our first construction is a generalization of the hypergraph $\mathcal{G}_1$ from \cite[Ch.~5]{A16}, which is a $4$-Ryser hypergraph with matching number $2$.
The edges of $\mathcal{G}_1$ are denoted using the following convention: fix an ordering of vertices on each of the $4$ sides and denote an edge as a vector of length $4$ where the $j^{th}$ coordinate denotes the position of the vertex of the $j^{th}$ side contained in the edge.
Using this notation, the edges of $\mathcal{G}_1$ are $\{1111, 1333, 1444, 5314, 6341, 6413, 2222, 2155, 3162, 4652, 4265, 2666, 2562, 1211\}$. 
We will denote the $i^{th}$ vertex on the $j^{th}$ side of the hypergraph $\mathcal{G}_1$ by $v_{ij}$.
As already pointed out in \cite[Ch.~5]{A16}, the two subhypergraphs of $\mathcal{G}_1$ whose edge sets are $\mathcal{F}_1 = \{1111, 1333, 1444, 5314, 6341, 6413\}$ and $\mathcal{F}_2 = \{2222, 2155, 3162, 4652, 4265, 2666 \}$, are edge-minimal $4$-Ryser hypergraphs that are infact subhypergraphs of a truncated projective plane $\mathrm{PG}(2, 3)$.
The edges of $\mathcal{G}_1$ that are not contained in these two subhypergraphs are $1211$ and $2562$. 
Note that $v_{12}$ is the only vertex common to the subhypergraphs on $\mathcal{F}_1$ and $\mathcal{F}_2$.

We now describe an isomorphism between $\mathcal{G}_1$ and a subhypergraph of our first construction. 
Map $v_{12}$ to the vertex $P$ and $v_{52}$ to the vertex $v$ of our hypergraph. 
The set of vertices \[\{v_{11}, v_{51}, v_{61}, v_{12}, v_{32}, v_{42}, v_{13}, v_{33}, v_{43}, v_{14}, v_{34}, v_{44}\}\] is mapped to the vertex set $\mathcal{P}_1'$ and the set \[\{v_{21}, v_{31}, v_{41}, v_{12}, v_{22}, v_{62}, v_{23}, v_{53}, v_{63}, v_{24}, v_{54}, v_{56}\}\] to $\mathcal{P}_2'$. 
Map the edge $1111$ to the line $\ell$.
It can be checked that $\mathcal{F}_1 \cup \{5431\}$ and $\mathcal{F}_2$ correspond to the edge sets $\mathcal{E}_1$ and $\mathcal{E}_2$ of our construction, where the conic $\mathcal{C}$ is given by the missing point in the second plane along with the points $v_{21}, v_{63}, v_{24}$ (it suffices to check that no three of these vertices are contained in a common edge).
The point $v$ corresponds to the point $v_{52}$, $e_1$ corresponds to $1211$ and $e_2$ corresponds to the edge $2562$.
This correspondence gives an isomorphism between $\mathcal{G}_1$ and a subhypergraph of our hypergraph.


\begin{thebibliography}{10}
\bibitem{A16} A.~Abu-Khazneh,  {\em  Matchings and covers of multipartite hypergraphs}, PhD thesis, The London School of Economics and Political Science (2016), \url{http://etheses.lse.ac.uk/3360/}.

\bibitem{ABPS19} A.~Abu-Khazneh, J.~Bar\'{a}t, A.~Pokrovskiy, T.~Szab\'{o}, A family of extremal hypergraphs for Ryser's conjecture, {\em J.~Combin.~Theory Ser.~A} 161 (2019), 164--177.

\bibitem{Aharoni} R.~Aharoni, Ryser’s conjecture for tripartite $3$-graphs, {\em Combinatorica} 21 (2001), 1--4.

\bibitem{ABW16} R.~Aharoni, J.~Bar\'{a}t, I.~Wanless, Multipartite hypergraphs achieving equality in Ryser's conjecture {\em Graphs Combin.} 32 (2016), 1--15.


\bibitem{BW18} D.~Best, I.M.~Wanless, What did Ryser conjecture? {\tt arXiv:1801.02893}, 2018.

\bibitem{B94} A.~Blokhuis, On the size of a blocking set in PG(2,p), {\em Combinatorica} 14 (1994), 111--114.

\bibitem{B96} A.~Blokhuis, D.~Mikl\'{o}s, V.T.~S\'{o}s, T.~Sz\"{o}ny, Blocking sets in Desarguesian planes, {\em Combinatorics, Paul Erd\"{o}s is eighty, Vol. 2} (Keszthely, 1993), Bolyai Soc. Math. Stud., 1996,  133--155.


\bibitem{BFK89} E.~Boros, Z.~F\"{u}redi, J.~Kahn, Maximal intersecting families and affine regular polygons in PG$(2,q)$, {\em J.~Combin.~Theory Ser.~A} 52 (1989),1--9.


\bibitem{B70}  A.A.~Bruen,  Blocking  sets  in  finite  projective  planes,
{\em SIAM  J.  Applied Math.} 21 (1971), 380--392.

\bibitem{BT77} A.A.~Bruen, J.A.~Thas, Blocking sets,
{\em Geom. Dedicata} 6 (1977), 193--203.

\bibitem{FHMW17} N.~Franceti\'{c}, S.~Herke, B.~D.~McKay, I.~M.~Wanless, On Ryser’s conjecture for linear intersecting multipartite hypergraphs, {\em European J.~Combin.~}61 (2017), 91--105.

\bibitem{HNS18} P.E.~Haxell, L.~Narins, T.~Szab\'{o},  Extremal hypergraphs for Ryser's conjecture, {\em J.~Combin.~Theory Ser.~A} 158 (2018), 492--547.

\bibitem{Haxell-Scott17} P.E.~Haxell, A.~D.~Scott, A note on intersecting hypergraphs with large cover number, {\em Electron.~J.~Combin.~}24 (2017), \#P3.26.

\bibitem{Henderson71} J.R.~Henderson, {\em Permutation Decompositions of $(0, 1)$-matrices and decomposition transversals}, PhD Thesis, Caltech (1971), pp.~60, \url{http://thesis.library.caltech.edu/5726/1/Henderson_jr_1971.pdf}.

\bibitem{Hbook} J.W.P.~Hirschfeld,  {\em Projective Geometries over Finite Fields}, New York: Oxford University Press, 1979.

\bibitem{Konig} D.~K\H{o}nig, Theorie der endlichen und unendlichen Graphen, Leipzig 1936.

\bibitem{Lovasz} L.~Lov\'{a}sz, {\em On Minimax Theorems of Combinatorics}, Ph.D thesis, Matemathikai Lapok 26 (1975), 209--264 (in Hungarian).

\bibitem{Segre} B.~Segre, Ovals in a finite projective plane, {\em Canadian Journal of Mathematics} 7 (1955), 414--416.

\bibitem{Tuza}  Zs.~Tuza, Ryser’s conjecture on transversals of $r$-partite hypergraphs, {\em Ars Combin.~16} (1983), 201--209.
\end{thebibliography}
\end{document}